\documentclass{cpamart1}     %% Input Standard CPAM class.
\received{Month 200X}       %\revised{date of revision}
\volume{000}
\startingpage{1}                      

\authorheadline{C. B. Muratov, M. Novaga, B. Ruffini}
\titleheadline{Charged flat drops}

\usepackage{amsmath,amsthm,amssymb,amsfonts,color,pifont,cite}
\usepackage[normalem]{ulem}
\usepackage{framed,enumerate}

\usepackage{mathptmx}
\newtheorem{theorem}{Theorem}[section]
\newtheorem{lemma}[theorem]{Lemma}
\newtheorem{corollary}[theorem]{Corollary}

\theoremstyle{definition}

\theoremstyle{remark}
\newtheorem{remark}[theorem]{Remark}

\newcommand{\eps}{\varepsilon} \def\R{\mathbb R} 
\def\I{\mathcal I} \def\E{\mathcal E} \def\diam{{\rm diam}}

\begin{document}                        %% Standard LaTeX command

%%      ---------------------------------------------------------------------
%%      -------------------------------- TITLE -----------------------------
%%      ---------------------------------------------------------------------

\title{On equilibrium shape of charged flat drops}

%%      ---------------------------------------------------------------------
%%      ------------------------------- AUTHORS -----------------------------
%%      ---------------------------------------------------------------------
\author{Cyrill B. Muratov}{New Jersey Institute of Technology}
% EXAMPLE: \author{Paris Hilton}{Universit‰ Paris-Sorbonne (Paris IV)}
% Uncomment the following lines as needed
\author{Matteo Novaga}{University
	of Pisa}
\author{Berardo Ruffini}{Institut Montpelli\'erain Alexander Grothendieck}
%\author{*** FOURTH AUTHOR'S NAME ***}{*** FOURTH AUTHOR'S AFFILIATION WHEN ARTICLE WAS WRITTEN ***}
%\author{*** FIFTH AUTHOR'S NAME ***}{*** FIFTH AUTHOR'S AFFILIATION WHEN ARTICLE WAS WRITTEN ***}
% Add additional names and affiliations as necessary using above format
%%      ---------------------------------------------------------------------
%%      --------------------------- DEDICATION  (OPTIONAL)------------------- 
%%      ---------------------------------------------------------------------

%       Uncomment the following line to insert a dedication.

%\dedication{ *** DEDICATION *** }        %% Enter dedication between braces.

%%      ---------------------------------------------------------------------
%%      --------------------------- ABSTRACT (OPTIONAL)----------------------
%%      ---------------------------------------------------------------------

%% ***** UNCOMMENT THE FOLLOWING TO INSERT AN ABSTRACT

\begin{abstract}
Equilibrium shapes of two-dimensional charged, perfectly conducting
liquid drops are governed by a geometric variational problem that
involves a perimeter term modeling line tension and a capacitary
term modeling Coulombic repulsion. Here we give a complete explicit
solution to this variational problem. Namely, we show that at fixed
total charge a ball of a particular radius is the unique global
minimizer among all sufficiently regular sets in the plane. For sets
whose area is also fixed, we show that balls are the only minimizers
if the charge is less than or equal to a critical charge, while for
larger charge minimizers do not exist. Analogous results hold for
drops whose potential, rather than charge, is fixed.
\end{abstract}

% With AMS-LaTeX, \maketitle follows the abstract
\maketitle   

%%      ---------------------------------------------------------------------
%%      ------------------- TABLE OF CONTENTS (OPTIONAL) --------------------
%%      ---------------------------------------------------------------------

%% ***** IF YOUR PAPER IS OVER 40 PAGES AND YOU WISH TO HAVE A TABLE
%% ***** OF CONTENTS, PLEASE UNCOMMENT THE FOLLOWING LINE

% \tableofcontents

%%      ---------------------------------------------------------------------
%%      ---------------------------- BODY OF PAPER --------------------------
%%      ---------------------------------------------------------------------

%%      Please input or insert the body of your paper here.

\section{Introduction}

In this paper we are interested in a geometric variational problem
that arises in the studies of electrified liquids (for a recent
overview, see, e.g., \cite{fernandezdelamora07}). Consider an
insulating fluid between two parallel plates, also made from an
insulating material. From a small opening in one of the plates an
immiscible, electrically conducting incompressible liquid is
introduced and is allowed to fill a region whose lateral dimensions
greatly exceed the distance between the plates. Because of its
essentially two-dimensional character, we call the resulting
pancake-shaped drop of the conducting liquid a {\em flat drop.} A
voltage is then applied through the opening, imparting an electric
charge to the liquid. One would like to understand the equilibrium
shape attained by the drop, either at constant voltage or constant
total charge, and either at constant or variable volume. Physically,
this problem is interesting because the electric field produced by the
charges tends to destabilize the liquid interface
\cite{Zel,Tay,FonFri,landau8,fernandezdelamora07}, an effect that is
used extensively in a variety of applications such as mass
spectrometry \cite{gaskell97}, electrospinning \cite{barrero07} and
microfluidics \cite{castrohernandez15}. Mathematically, this problem
is interesting because the interplay between the capillary and
Coulombic forces produces a competition that makes prediction of the
energy minimizing configurations a highly non-trivial task.

The equilibrium configurations of electrified conducting liquids may
be studied with the help of a basic variational model that goes back
all the way to Lord Rayleigh \cite{Ray}. Let $K \subset \R^3$ be a
compact set with smooth boundary. If $\sigma$ is the surface tension,
$\epsilon$ is the relative permittivity of the surrounding medium and
$\epsilon_0$ is the permittivity of vacuum, then the energy of a
charged, perfectly conducting liquid drop occupying $K$, with volume
$V$ and the total charge $Q$, is given by \cite{landau8}
\begin{align}
\label{eq:E3dQ}
E^Q_{3d}(K) = \sigma \mathcal H^2(\partial K) + {Q^2 \over 2 C(K)},
\end{align}
where $\mathcal H^s$ denotes the $s$-dimensional Hausdorff measure,
with $\mathcal H^2(\partial K)$ representing the surface area of the
liquid, and
\begin{align}
\label{eq:COm}
{1 \over C(K)} := {1 \over 4 \pi 
	\epsilon \epsilon_0} \inf_{\mu(K) = 1} \int_K \int_K
{d \mu(x) d \mu(y) \over |x - y|},
\end{align}
is the inverse electric capacitance of $K$ in the SI units. Here $\mu$
is a probability measure supported on $K$, and the infimum in
\eqref{eq:COm} is attained by a unique measure that concentrates on
$\partial K$ \cite{landkof}. If instead the voltage $U$ on
the drop is prescribed, then the total energy (including that of the
charge reservoir) is 
\begin{align}
\label{eq:E3dU}
E^U_{3d}(K) := \sigma \mathcal H^2(\partial K) - \frac12 C(K) U^2.
\end{align}

For flat drops, this model can be reformulated as follows. Let $d$ be
the distance between the plates and assume that
$K = \Omega \times [0, d]$, where $\Omega \subset \mathbb R^2$ is a
compact set with smooth boundary which represents the shape of the
flat drop. Note that this means that we are assuming the contact angle
$\theta = \frac{\pi}{2}$, corresponding to a flat meniscus. For a
general contact angle $\theta \in [0, \pi]$ the set $K$ may be
similarly defined in terms of $\Omega$, with the shape of the meniscus
determined by balancing the capillary forces locally. In this case the
meniscus cross-section will have the shape of a circular arc of length
$\ell = {\pi - 2 \theta \over 2 \cos \theta} \, d$, with $\ell = d$ if
$\theta = \frac{\pi}{2}$. Then, to the leading order in $d \ll L$,
where $L$ is the characteristic dimension of the drop, the energy
$E_{3d}^Q(K)$ is given by
\begin{align}
\label{eq:EOmQ}
E_{2d}^Q(\Omega) := \sigma \ell 
\mathcal H^1(\partial \Omega) + {Q^2 \over 8 \pi \epsilon
	\epsilon_0} \I_1(\Omega), 
\end{align}
where we introduced the Riesz capacitary energy of a set:
\begin{equation}\label{Ialpha}
\I_\alpha(\Omega):=\inf\left\{
\int_\Omega \int_\Omega \frac{d\mu(x)\,d\mu(y)}{|x-y|^\alpha} \ :\
\mu(\Omega)=1 \right\},
\end{equation}
which is more generally defined for all Borel sets
$\Omega \subset \R^N$ and $\alpha \in (0, N)$. Similarly,
$E_{3d}^U(K)$ becomes
\begin{align}
\label{eq:EOmU}
E_{2d}^U(\Omega) := \sigma \ell 
\mathcal H^1(\partial \Omega) - {2 \pi \epsilon
	\epsilon_0 U^2 \over \I_1(\Omega)}.
\end{align}
The volume of the flat drop is, to the leading order, equal to
$V = \mathcal |\Omega| d$, where $|\Omega|$ denotes the area of
$\Omega$. Hence, fixing the volume of $K$ amounts to fixing
$|\Omega|$.

We note that the infimum in the definition of $\I_1(\Omega)$ is
attained by a unique measure supported on $\Omega$, and the value
of $\I_1(\Omega)$ is proportional to the inverse of the Riesz capacity
of $\Omega$ \cite{landkof}. Also, for fixed volume there is no
shape-dependent contribution to the energy from the liquid-solid
interface to the leading order.  We point out that the same type of
mathematical models is relevant to systems which give rise to electric
charges concentrating on planar regions, such as high-$T_c$
superconductors \cite{emery93}, Langmuir monolayers \cite{andelman87}
and, possibly, graphene \cite{samaddar16}. In addition, it should be
possible to rigorously obtain the corresponding two-dimensional
expressions for the energy from the original three-dimensional energy
in the limit $d \to 0$ via $\Gamma$-convergence for a suitable notion
of convergence of sets \cite{braides}.

We are now in a position to introduce the dimensionless versions of
the energies that we will analyze. Setting
\begin{align}
\label{eq:EElam}
\E^Q_\lambda(\Omega) := {E^Q_{2d} (L \Omega) \over \sigma \ell L},
\qquad \lambda := {Q^2 \over 8 \pi \epsilon \epsilon_0 \sigma \ell
	L^2}, 
\end{align}
we obtain a one-parameter family of energy functionals
\begin{align}
\label{eq:Elam}
\E^Q_\lambda(\Omega) = \mathcal H^1(\partial \Omega) + \lambda
\I_1(\Omega), 
\end{align}
with $\lambda > 0$ a fixed dimensionless parameter, defined on all
compact subsets of $\R^2$. Similarly, setting
\begin{align}
\label{eq:EElambda2}
\E^U_\lambda (\Omega) := {E^U_{2d} (L \Omega) \over \sigma \ell L},
\qquad \lambda := {2 \pi \epsilon \epsilon_0  U^2 \over \sigma \ell}, 
\end{align}
we obtain a one-parameter family of energy functionals
\begin{align}
\label{eq:Elambda2}
\E^U_\lambda(\Omega) = \mathcal H^1(\partial \Omega) - {\lambda \over
	\I_1(\Omega)}, 
\end{align}
with $\lambda > 0$ a fixed dimensionless parameter. These energy
functionals and their global minimizers with or without constraints
are the main subject of the present paper.

There has recently been a great interest in non-local variational
problems involving energies like those in \eqref{eq:Elam} and
\eqref{eq:Elambda2}.  In particular, the non-local isoperimetric
problems governed by \eqref{eq:E3dQ}, or \eqref{eq:Elam}, or their
generalizations involving Riesz energy, in which the measure $\mu$ is
fixed to be the uniform measure, was recently treated in
\cite{KM1,KM2, lu14, bonacini14,julin14,
	mz:agag15,figalli15,frank15,FKN} (see also
\cite{alberti09,choksi10, gms:arma13,gms:arma14, kmn16} for closely
related problems on large bounded domains; these lists are by far
not meant to be exhaustive). One of the motivations for those studies
comes from the fact that the three-dimensional problem with uniform
charge density corresponds to the famous model of the atomic nucleus
introduced by Gamow \cite{gamow30} and is used to explain the
stability of nuclear matter and nuclear fission
\cite{weizsacker35,bohr39}. It is now known that in the case of the
energy in \eqref{eq:Elam} with $\I_1$ evaluated, using
$d \mu = |\Omega|^{-1} dx$, the unique minimizer with $|\Omega| = m$
fixed is a ball as long as $m \leq m_0 \lambda^{-1}$, where $m_0 > 0$
is a universal constant. On the other hand, there is no minimizer with
$|\Omega| = m$ as soon as $m > m_2 \lambda^{-1}$, where $m_2 \geq m_0$
is another universal constant \cite{KM1}. These conclusions have been
extended to the case of higher-dimensional problems and more general
Riesz kernels \cite{KM2,lu14,julin14,figalli15}. The basic intuition
behind these results is that when the value of $m$ is sufficiently
small, the perimeter term dominates the energy, forcing the minimizer
to coincide with the one for the classical isoperimetric problem. On
the other hand, when the value of $m$ is sufficiently large, the
repulsive term becomes dominant, forcing the charges to move off to
infinity.  The observations above motivate the conjecture, first
stated in \cite{choksi11}, that minimizers of the problem with uniform
charge density are balls, unless it is better to split one ball into
to two equal size balls. Whether or not this conjecture is true in its
full generality is a very difficult question, but it was answered in
the affirmative in the case of two-dimensional problems involving
Riesz kernels with a sufficiently small exponent $\alpha$ \cite{KM1,
	mz:agag15}. A weaker version of this result in higher dimensions,
namely, that minimizers of the above problem with $\alpha$
sufficiently small are balls whenever they exist, without specifying
the threshold for existence, was established in \cite{bonacini14}.

One might expect that the capacitary problem in \eqref{eq:E3dQ} should
exhibit a similar behavior. Studies of instability of balls with
respect to small smooth perturbations in three space dimensions go
back to Lord Rayleigh \cite{Ray}, and emergence of regular
non-symmetric critical points out of the bifurcation at the critical
charge has been established much more recently in \cite{FonFri} (for
some numerical work, see also
\cite{basaran89a,burton11,garzon14}). Yet, a striking result from a
recent paper \cite{gnrI} shows that the problem of minimizing the
energy in \eqref{eq:E3dQ} {\em never} admits global minimizers with
$|\Omega| = m$. Even worse, this problem was shown never to admit any
regular {\em local minimizers}, understood in any reasonable sense
\cite{MurNov}. This severe ill-posedness of the three-dimensional
capacitary problem has to do with an incompatibility between the
perimeter term, which sees two-dimensional sets, and the Coulombic
term, which sees sets of positive Newtonian capacity. The latter,
however, could concentrate on sets of positive $s$-dimensional
Hausdorff measure with $1 < s < 2$ \cite{landkof}. Thus, there is a
gap between the Hausdorff dimension of the sets that are negligible
with respect to the perimeter and those negligible with respect to the
Coulombic energy, which prevents the perimeter from imparting
classical regularity to the minimizers. Note that existence may be
restored, if one looks for minimizers in more regular classes, such
as, for example, the class of convex compact sets \cite{gnrII}. Also,
well-posedness under confinement in three dimensions can be restored
by including entropic effects into the model, resulting in a finite
screening length \cite{MurNov}.

Coming back to our main problem of interest, which consists of
minimizing $\E^Q_\lambda(\Omega)$ in \eqref{eq:Elam} among all compact
sets $\Omega \subset \R^2$ with $|\Omega| = m$ fixed, one can see that
the reduced dimensionality of the problem makes it a borderline case
between the ill-posed capacitary problem in three-dimensions and the
well-posed uniform charge density problem in two dimensions. Indeed,
in this case the perimeter controls sets of zero Riesz capacity for
$\alpha = 1$ in two space dimensions. Therefore, the question
of existence vs. non-existence of minimizers in this setting becomes
particularly delicate. In what follows, we will show that the
non-local isoperimetric problem associated with \eqref{eq:Elam}
remains on the ``good side'' of the critical value of $\alpha = 1$,
exhibiting existence of minimizers for small values of $m$ and
non-existence for large values of $m$.

Remarkably, we are able to give a complete explicit solution to this
problem, both with and without the constraint on $|\Omega|$. In the
absence of the constraint, the energy $\E^Q_\lambda$ is uniquely
minimized by balls of radius $R = \frac12 \sqrt{\lambda}$. With the
constraint $|\Omega| = m$ the unique minimizers of $E^Q_\lambda$ are
balls when $\lambda \leq 4 m / \pi$, while there are no minimizers
when $\lambda > 4 m / \pi$. We note that balls loose their minimizing
property {\em before} it becomes energetically favorable to split one
ball into two well separated balls with equal area and charge, which
happens for $\lambda > 4 m \sqrt{2} / \pi$, or before deforming a ball
into an elliptical domain of the same area lowers energy, which
happens for $\lambda > 12 m / \pi$. The nature of the instability for
$\lambda > 4 m / \pi$ is similar to the one for the three-dimensional
capacitary problem \cite{gnrI}: one can lower the energy by splitting
off a large number of small balls, putting the right amount of charge
on them and sending them off to infinity. This {\em catastrophic} loss
of stability suggests that the ill-posedness exhibited by the
three-dimensional problem \cite{gnrI,MurNov} begins to manifest itself
in our two-dimensional problem at a critical value of charge. Lastly,
we also provide a complete solution to the problem of minimizing
$\E^U_\lambda$ from \eqref{eq:Elambda2}. Without the constraint on
$|\Omega|$, this energy has minimizers if and only if
$\lambda = \pi^2$, and the only minimizers are balls (of any
radius). If also $|\Omega| = m$, then balls of area $m$ are the only
minimizers when $\lambda \leq \pi^2$, and there are no minimizers for
$\lambda > \pi^2$.

The strategy of the proof of our main theorems is to take advantage of
the fact that without the volume constraint the energy $\E^Q_\lambda$
of a set $\Omega$ decreases if we convexify each connected
component of $\Omega$. Then we exploit the fact that the
perimeter is {\em linear} for the Minkowski sum \cite{gardner}, while
the non-local term $\I_1$ is {\em sublinear} \cite{NovRuf}. This,
together with some classical, simple results in the theory of convex
bodies, allows us to transform each convex component of a set
into a ball. Eventually we show that by suitably merging several
balls into one ball we decrease the energy. Let us stress that such a
technique is very special and applies only to the two-dimensional case
and only to the capacitary energy $\I_1$. It is indeed an open problem
whether $\I_\alpha$ is a subadditive operator with respect to the
Minkowski summation for $\alpha\ne 1$ (and for $\alpha\ne N-2$ in $\R^N$
with $N \geq 3$). For the constrained problem, we combine the previous
result with the well-known {\em maximizing} property of balls with
respect to the Riesz capacitary energy \cite{polya1,polya2} to prove
existence for small charges. Finally, the non-existence for large
charges for the constrained problem is obtained by constructing a test
configuration with many balls, whose energy in the limit coincides
with that of the unconstrained problem.

Let us also point out that while our paper is concerned exclusively
with global energy minimizers, existence of local minimizers or
critical points of the energy is naturally an interesting question,
both physically and mathematically. Care, however, is needed in
defining the appropriate notion of solution in view of the a priori
lack of regularity and potentially severe ill-posedness of the
considered variational problems \cite{MurNov}. In particular, one
needs to take into consideration the possibility of either local
minimizers or critical points consisting of sets with many small holes
(``froth''). Some new quantitative insights into the capacitary energy
of sets of finite perimeter would be needed to deal with those issues,
which goes beyond the scope of the present work.

The rest of the paper is organized as follows: in Section
\ref{sec:main-results}, we give the precise definitions of the
variational problems to be studied and state our main results, in
Section \ref{sec:prelim} we discuss some preliminary results, and in
Section \ref{sec:proofs} we present the proofs of the main theorems.

\section{Setting and statements of the main results}
\label{sec:main-results}

In this section we state the main results of our paper. Viewing
$\E^Q_\lambda$ in \eqref{eq:Elam} as a map from $\mathcal B(\R^2)$ to
$[0, +\infty]$, where $\mathcal B(\R^2)$ is the family of all Borel
subsets of $\R^2$, we start by defining the appropriate admissible
classes in the domain of $\E^Q_\lambda$. Given two sets
$A,B\subset\R^2$ we introduce the equivalence relation $A\sim B$ if
$\mathcal H^1(A\setminus B)=\mathcal H^1(B\setminus A)=0$.  We then
set
\begin{align}
\label{eq:A}
\mathcal A := \{ \Omega \subset \R^2 \ : \ \Omega \ \text{compact},
\ \ |\Omega| > 0, \ \ \mathcal H^1(\partial \Omega) < \infty \}/\sim\,.
\end{align}
Notice that the equivalence relation $\sim$ in the definition of the
class $\mathcal A$ will be needed in order to have uniqueness of
minimizers of $\E^Q_\lambda$ over $\mathcal A$.

The definition above ensures that $0 < \E^Q_\lambda(\Omega) < \infty$
whenever $\Omega \in \mathcal A$. We note that physically the class
$\mathcal A$ corresponds to flat drops with fixed charge and variable
volume, which can be achieved by placing a porous membrane, permeable
to the liquid but impermeable to the charge-carrying ions, into the
opening supplying the liquid. Similarly, for $m > 0$ we define
\begin{align}
\label{eq:Am}
\mathcal A_m := \{ \Omega \in \mathcal A \ : \ |\Omega| = m\}.
\end{align}
Note that the definitions of both the energy and the associated
admissible classes $\mathcal A$ and $\mathcal A_m$ involve the
Hausdorff measure $\mathcal H^1$ of the {\em topological boundary}.
Unfortunately, $\mathcal H^1$ fails to be lower semicontinuous even
with respect to Hausdorff convergence of compact sets, making it
difficult to apply the direct method of the calculus of variations to
establish existence of minimizers of $\E^Q_\lambda$. One possibility
to proceed might be to replace $\mathcal H^1(\partial \Omega)$ with
$P(\Omega)$, where
\begin{equation}\label{caccioper}
P(\Omega) := \sup\left\{\int_\Omega \nabla \cdot \phi \, dx \ : \ 
\phi\in C^\infty_c(\R^2,\R^2),\,\,\|\phi\|_{L^\infty(\R^2)}\le 1
\right\},  
\end{equation}
is the perimeter functional (which coincides with
$\mathcal H^1(\partial \Omega)$ for sets with smooth boundary
\cite{AFP}), and work in the class of sets of finite perimeter.  This,
however, would bring us to an ill-posed problem: since $\I_1$ is a
capacitary term, it is not well defined on sets which are defined up
to sets of Lebesgue measure zero \cite{landkof}.  This means, for
example, that it is possible to find a compact set consisting of a
disjoint union of a closed ball of area $m$ and a closed set of
Lebesgue measure zero but positive Riesz capacity (by \cite[Theorem
3.13]{landkof}, for the latter one can take any compact set
$K \subset \R^2$ such that $0 < \mathcal H^s(K) < \infty$ for some
$1 < s < 2$).  Putting all the charge on the latter and then
translating and dilating it off to infinity, one recovers the
perimeter of the ball of area $m$ as the infimum of the energy. Thus
the infimum would clearly not be attained, and this is the reason why
we restrict ourselves to the classes $\mathcal A$ and $\mathcal A_m$.

Notice that, for all $\Omega\in\mathcal A$, we have
$P(\Omega)\le \mathcal H^1(\partial\Omega)$.  Indeed, following
\cite{AFP} we have that $P(\Omega)=\mathcal H^1(\partial^M\Omega)$
where
\begin{align}
\partial^M\Omega:=\left\{x\in \R^2 \ : \ \lim_{r\to
	0}\frac{|\Omega\cap 
	B_r(x)|}{\pi r^2}\not\in\{0,1\}\right\}\subseteq \partial\Omega. 
\end{align}
However, $P(\Omega)= \mathcal H^1(\partial\Omega)$ as soon as $\Omega$
has Lipschitz boundary.

\medskip

In this paper we do not rely on the direct method of the calculus of
variations. Instead, we directly show that balls are minimizers for
our problem.  We also point out that by the results of \cite{gnrI},
the infimum of the energy over $\mathcal A_m$ is {\em never attained},
if $\I_1$ is replaced by $\I_\alpha$ with $\alpha \in (0, 1)$,
contrary to the case of $\alpha = 1$ considered here.

\medskip

%\medskip

We now state our first result concerning minimizers of $\E^Q_\lambda$
over $\mathcal A$.

\begin{theorem}\label{t:A}
	Let $\lambda > 0$ and let $R_\lambda = \frac12 \sqrt{\lambda}$. Then
	the closed ball $\overline B_{R_\lambda}$ is the unique (up to
	translations) minimizer of $\E^Q_\lambda$ over $\mathcal A$.
\end{theorem}

We note that the optimal radius $R_\lambda$ can be determined by
minimizing over balls, for which the energy is explicitly (see Lemma
\ref{l:ell})
\begin{align}
\label{eq:EBR}
\E^Q_\lambda(\overline B_R) = 2 \pi R + {\lambda \pi \over 2 R} \geq
2 \pi \sqrt{\lambda} = \E^Q_\lambda(\overline B_{R_\lambda}),
\end{align}
with the lower bound in the right-hand side attained only if
$R = R_\lambda$. We also note that because of the characterization of
sets of finite perimeter in the plane \cite{ACMM}, the statement of
Theorem \ref{t:A} remains true when the minimization is carried out
among all  sets of finite perimeter,
provided that the set $\Omega$ is identified with its
measure-theoretic closure. Alternatively, the statement holds when
minimizing $\E^Q_\lambda$ among convex compact sets. We caution the
reader, however, that not every minimizing sequence for $\E^Q_\lambda$
(after suitable translations) has a ball of radius $R_\lambda$ as the
limit.

\begin{remark}
	\label{r:A}
	For $\lambda > 0$ the infimum of $\E^Q_\lambda$ over $\mathcal A$ is
	achieved by a sequence of $N$ balls of radius $R_\lambda / N$
	running off to infinity. There are also minimizing sequences
	consisting of $N$ balls with $N \to \infty$.
	Indeed, taking
	$\Omega_\eps = \bigcup_{i=1}^N \overline B_{R_\lambda / N}(x_i)$, with
	$|x_i - x_j| \gg R_\lambda$, in view of Theorem \ref{t:A} and Lemma
	\ref{l:partit} we have
	\begin{align}
	\label{eq:EQeps}
	\E^Q_\lambda(\Omega_\eps) \leq N \E^Q_{\lambda / N^2}(\overline
	B_{R_\lambda / N}) + \eps = \E^Q_\lambda(\overline B_{R_\lambda})
	+ \eps = \inf_{\Omega \in \mathcal A} \E^Q_\lambda(\Omega) + \eps,  
	\end{align}
	for an arbitrary $\eps > 0$. Thus, the energy can be equally minimized
	by a fine ``mist'' of droplets moving off to infinity by distributing
	the charge proportionally to their perimeter.
\end{remark}

As a byproduct of the proof of Theorem \ref{t:A} we can show that
balls minimize $\I_1$ (that is, maximize the capacity) among sets of
fixed perimeter, thus generalizing a result in \cite[Corollary
3.2]{NovRuf}.

\begin{corollary}\label{t:B}
	The closed ball $\overline B_{R}$ is the unique (up to translations)
	minimizer of $\I_1$ among the sets $\Omega\in \mathcal A$ with
	$P(\Omega)=2\pi R$.  The ball $\overline B_{R}$ is also the unique
	minimizer of $\I_1$ among the sets $\Omega\in \mathcal A$ with
	$\mathcal H^1(\partial\Omega)=2\pi R$.
\end{corollary}

Turning now to the minimizers of $\E^Q_\lambda$ over $\mathcal A_m$,
we have the following result. 

\begin{theorem}\label{t:Am}
	Let $\lambda > 0$, $m > 0$, let $R = \sqrt{m / \pi}$, and define
	\begin{align}
	\label{eq:lam0}
	\lambda_0^Q := {4 m \over \pi}.
	\end{align}
	Then:
	\begin{enumerate}[(i)]
		\item The closed ball $\overline B_R$ is the unique (up to
		translation) minimizer of $\E^Q_\lambda$ over $\mathcal A_m$, if
		$\lambda \leq \lambda_0^Q$.
		\item There is no minimizer of $\E^Q_\lambda$ over $\mathcal A_m$,
		if $\lambda > \lambda_0^Q$.
	\end{enumerate}
\end{theorem}

One should contrast the result in Theorem \ref{t:Am} with the
expectation in the case of nonlocal isoperimetric problems with
uniform charge that balls are the unique minimizers, unless it becomes
advantageous to split a ball into two equal size balls
\cite{choksi11,KM1,mz:agag15}. In our problem the latter would happen
at the value of $\lambda$ at which two balls of area $m/2$ infinitely
far apart and carrying equal charge have the same total energy as one
ball of area $m$ and the same total charge. By Lemma \ref{l:lamc1c2},
this is the case if $\lambda > \lambda_{c1}^Q$, where
\begin{align}
\label{eq:lamc1}
\lambda_{c1}^Q := {4 m \sqrt{2} \over \pi}.
\end{align}
Comparing this value with $\lambda_0^Q$, one sees, however, that balls
become unstable {\em before} it becomes energetically favorable to
split them into two equal balls. Furthermore, by Lemma \ref{l:lamc1c2}
the instability with respect to elongations occurs at even higher
values of $\lambda > \lambda_{c2}^Q$, where
\begin{align}
\label{eq:lamc2}
\lambda_{c2}^Q := {12 m \over \pi}.
\end{align}
An inspection of the proof of Theorem \ref{t:A} shows that the
nature of the instability of balls at $\lambda = \lambda_0^Q$ is, in
fact, very different from the two just described: as soon as
$\lambda > \lambda_0^Q$, it becomes energetically favorable to split
off many small balls and move them far apart. Those balls will have a
small total area, but carry a finite fraction of the total
charge. This nonlinear instability is closely related to the
ill-posedness of the three-dimensional capacitary problem
\cite{gnrI,MurNov}.

Observe that, as with minimizing $\E^Q_\lambda$ over $\mathcal A$, the
infimum over $\mathcal A_m$ for $\lambda \leq \lambda_0^Q$ may also be
achieved by a sequence of many balls moving off to infinity, as in
Remark \ref{r:A}. By \eqref{eq:EBR}, in this case we have
\begin{align}
\label{eq:EinfAm}
\inf_{\Omega \in \mathcal A_m} \E^Q_\lambda(\Omega) = \sqrt{4 \pi m}
+ \lambda \pi \sqrt{\pi \over 4 m}  \qquad \forall \lambda \leq
\lambda_0^Q. 
\end{align}
On the other hand, when the infimum is not attained, we can still
characterize the optimal scaling of the energy in terms of $\lambda$.

\begin{theorem}
	\label{t:scale}
	Let $m > 0$ and let $\lambda > \lambda_0^Q$, where $\lambda_0^Q$ is
	defined in \eqref{eq:lam0}. Then
	\begin{align}
	\label{eq:EQm}
	\inf_{\Omega \in \mathcal A_m} \E^Q_\lambda(\Omega)
	= 2\pi \sqrt{\lambda}.
	\end{align}
\end{theorem}
\noindent Thus, we have
$\inf_{\Omega \in \mathcal A_m} \E^Q_\lambda(\Omega) \sim \max \left\{
\sqrt{m}, \sqrt{\lambda} \right\}$.
Note that in physical terms this result means that the minimal energy
is linear in charge for all sufficiently large charges. In this case
the infimum of the energy is achieved by a union of a ball carrying
the fraction $\sqrt{\lambda_0^Q/\lambda}$ of the total charge
whose area converges to $m$ and many small balls with vanishing total
area and carrying the fraction
$\sqrt{(\lambda - \lambda_0^Q) / \lambda}$ of the total charge, far
apart.

We conclude by turning our attention to the problem associated with
the energy $\E^U_\lambda$ from \eqref{eq:Elambda2}. One key difference
between $\E^U_\lambda$ and $\E^Q_\lambda$ is that the former is not
bounded from below a priori.  Consider the energy of a ball of radius $R$, 
which equals
\begin{align}
\label{eq:EUBR}
\E^U_\lambda(\overline B_R) = 2 \pi R - {2 \lambda R \over \pi},
\end{align}
and observe a marked difference between this expression and the one in
\eqref{eq:EBR}, which has a unique minimum as a function of $R$. Here,
instead, the energy is a one-homogeneous function of $R$ vanishing at
$\lambda = \lambda_0^U$, where $\lambda_0^U := \pi^2$. At
$\lambda = \lambda_0^U$ any ball is a minimizer of the energy (among
balls), while for $\lambda < \lambda_0^U$ the infimum of the energy is
attained by a sequence of shrinking balls, and for
$\lambda > \lambda_0^U$ the energy is not bounded from below. This
picture remains valid also among all sets in $\mathcal A$.

\begin{theorem}
	\label{t:EUA}
	Let $\lambda > 0$ and define 
	\begin{align}
	\label{eq:lam0p}
	\lambda_0^U := \pi^2.     
	\end{align}
	Then:
	\begin{enumerate}[(i)]
		\item If $\lambda < \lambda_0^U$, we have
		$\inf_{\Omega \in \mathcal A} \E^U_\lambda(\Omega) = 0$, and the
		infimum is not attained.
		\item If $\lambda = \lambda_0^U$, then
		$\inf_{\Omega \in \mathcal A} \E^U_\lambda(\Omega) = 0$, and the
		infimum is attained by any closed ball and no other set.
		\item If $\lambda > \lambda_0^U$, we have
		$\inf_{\Omega \in \mathcal A} \E^U_\lambda(\Omega) = -\infty$.
	\end{enumerate}
\end{theorem}

Similarly, in the case of minimizing over $\mathcal A_m$ we have the
following result.

\begin{theorem}
	\label{t:EUAm}
	Let $\lambda > 0$, $m > 0$, let $R = \sqrt{m / \pi}$, and let
	$\lambda_0^U$ be defined in \eqref{eq:lam0p}. Then:
	\begin{enumerate}[(i)]
		\item If $\lambda \leq \lambda_0^U$, then $\overline B_R$ is the
		unique minimizer (up to translations) of $\E^U_\lambda$ over
		$\mathcal A_m$.
		\item If $\lambda > \lambda_0^U$, then
		$\inf_{\Omega \in \mathcal A} \E^U_\lambda(\Omega) = -\infty$.
	\end{enumerate}
\end{theorem}
\noindent In some sense, the problem associated with $\E^U_\lambda$ is
simpler, because both the perimeter and the capacitary term are
one-homogeneous functions with respect to dilations, making the
problem very degenerate. 

\section{Preliminary results}\label{sec:prelim}

\subsection{Basic facts about the capacitary energy}
\label{sec:some-facts-about}

We start by collecting several results about existence and uniqueness
of equilibrium measures on admissible sets.

\begin{lemma}
	\label{l:optm}
	Let $\Omega \in \mathcal A$. Then there exists a unique probability
	measure $\mu$ {over $\R^2$} supported on $\Omega$ such that
	\begin{align}
	\I_1(\Omega) = \int_\Omega \int_\Omega {d \mu(x)  d \mu(y) \over
		|x - y|}.
	\end{align}
	Furthermore, $\mu(\partial \Omega) = 0$, and we have
	$d \mu(x) = \rho(x) dx$ for some
	$\rho \in L^1(\Omega)$ satisfying
	$0 < \rho(x) \leq C / \mathrm{dist} \, (x, \partial \Omega)$ for
	some constant $C > 0$ and all $x \in \mathrm{int}(\Omega)$.
\end{lemma}

\begin{proof}
	For the existence and uniqueness of the equilibrium measure
	$\mu$ we refer to \cite[pp. 131--133]{landkof}.  By \cite[Theorems 3.13 and
	3.14]{landkof}, we have $\mu(\Sigma)=0$ for any set $\Sigma$ such
	that $\mathcal H^1(\Sigma)<\infty$, so that in particular
	$\mu(\partial\Omega)=0$.
	
	To show that the equilibrium measure $\mu$ is absolutely continuous
	with respect to the Lebesgue measure in the plane and has a strictly
	positive density, we recall that the Riesz potential
	$v: \R^2 \to [0, \infty]$ associated with $\Omega$:
	\begin{align}
	\label{eq:vriesz}
	v(x) := \int_\Omega {d \mu(y) \over |x - y|},
	\end{align}
	is equal to $\I_1(\Omega)$ in $\Omega$ up to a set of zero 1-Riesz
	capacity and satisfies $v(x) \leq \I_1(\Omega)$ for all $x \in \R^2$
	\cite[Theorem 2.6 and page 137]{landkof}. Therefore, by
	\cite[Theorem 3.13]{landkof} we have $v(x) = \I_1(\Omega)$ for
	a.e. $x \in \Omega$.
	
	\noindent Let now $\phi \in C^\infty_c(\R^2)$ and
	$\psi \in C^\infty(\R^2) \cap L^1(\R^2)$ be such that
	$(-\Delta)^{1/2} \phi = \psi$, which is equivalent to \cite[Lemma
	3.2 and Proposition 3.3]{dinezza}
	\begin{align}
	{1 \over 4 \pi} \int_{\R^2} {2 \phi(x) - \phi(x - y) - \phi(x + y)
		\over |y|^3} \, dy = \psi(x) \qquad \forall x \in \R^2.
	\end{align}
	Multiplying \eqref{eq:vriesz} by $\psi$ and integrating over $\R^2$,
	from Fubini-Tonelli theorem we obtain
	\begin{align}
	\label{eq:frac}
	{1 \over 4 \pi} \int_{\R^2} \int_{\R^2}  {v(x) (2 \phi(x) - \phi(x
		- y) - \phi(x + y)) \over |y|^3} \, dx \, dy = 2 \pi \int_\Omega
	\phi(x) \, d \mu(x),
	\end{align}
	where we also took into account that
	$\phi(x) = {1 \over 2 \pi} \int_{\R^2} {\psi(y) \over |x - y|} \,
	dy$
	for every $x \in \R^2$ \cite[Lemma 1.3]{mazyakhavin}. Note that
	\eqref{eq:frac} is the distributional form of the equation
	$(-\Delta)^{1/2} v = 2 \pi \mu$. Finally, observe that for any
	$\phi$ with support in the interior of $\Omega$ we have
	\begin{align}
	\label{eq:integrand}
          {1 \over 4 \pi} 
          &\int_{\R^2} \int_{\R^2}  {v(x) (2 \phi(x) - \phi(x
            - y) - \phi(x + y)) \over |y|^3} \, dx \, dy \\
          &= \lim_{r \to 0}   {1
            \over 4 \pi} \int_{\R^2 \backslash B_r(0)} \left( \int_{\R^2}
            {v(x) (2 \phi(x) - \phi(x  
            - y) - \phi(x + y)) \over |y|^3} \, dx \right) \,dy \notag \\
          &=  \lim_{r \to 0}   {1
            \over 4 \pi} \int_{\R^2 \backslash B_r(0)} \left( \int_{\R^2}
            {\phi(x) (2 v(x) - v(x  
            - y) - v(x + y)) \over |y|^3} \, dx \right) dy \notag \\
          &=  {1
            \over 4 \pi} \int_\Omega \phi(x) \left( \int_{\R^2}  {2 v(x) - v(x 
            - y) - v(x + y) \over |y|^3} \, dy \right)\, dx \notag
          \\
          &=: \int_\Omega
            \phi(x) \rho(x) \, dx, \notag
	\end{align}
	where we took into account that the inner integrand in the last line
	of \eqref{eq:integrand} is zero for a.e. $y \in B_r(0)$ with
	$r := \mathrm{dist}(\mathrm{supp} \, \phi, \partial
	\Omega)$.
	Thus $d \mu(x) = \rho(x) dx$, and the rest of the claims follows
	from the definition of $\rho$, noting that the integrand in the last
	line of \eqref{eq:integrand} is strictly positive on a set of
	positive measure and is bounded above by
	$2 \I_1(\Omega) / |y|^3$ for all
	$y \in \R^2 \backslash B_r(0)$.
\end{proof}

We next state a result about the capacitary energies of unions
of disjoint sets from $\mathcal A$.

\begin{lemma}
	\label{l:partit}
	Let $\Omega_1 \in \mathcal A$ and $\Omega_2 \in \mathcal A$ be two
	disjoint sets, and let $\theta \in [0,1]$. Then
	\begin{align}\label{stuno}
	\I_1(\Omega_1 \cup \Omega_2) \leq \theta^2 \I_1(\Omega_1) + (1 -
	\theta)^2 \I_2(\Omega_2) + {2 \theta (1 - \theta) \over
		\mathrm{dist} \, (\Omega_1, \Omega_2)}.
	\end{align}
	Moreover, there exists $\bar \theta\in (0,1)$ such that 
	\begin{align}\label{stdue}
	\I_1(\Omega_1 \cup \Omega_2) > \bar\theta^2 \I_1(\Omega_1) + (1 -
	\bar\theta)^2 \I_2(\Omega_2).
	\end{align}
\end{lemma}

\begin{proof}
  Let $\theta \in [0,1]$, and let $\mu_1$ and $\mu_2$ be the
  equilibrium measures for $\I_1(\Omega_1)$ and $\I_1(\Omega_2)$,
  respectively. Then $\tilde\mu = \theta \mu_1 + (1 - \theta) \mu_2$
  is a probability measure supported on $\Omega_1\cup\Omega_2$. Hence
	\begin{align}
          \I_1(\Omega_1 \cup \Omega_2) \leq \int_{\Omega_1 \cup \Omega_2}
          \int_{\Omega_1 \cup \Omega_2} {d 
          \tilde\mu(x) \, d \tilde\mu(y) \over |x - y|} = \theta^2
          \int_{\Omega_1} \int_{\Omega_1} {d 
          \mu_1(x) \, d \mu_1(y) \over |x - y|} \\
          + (1 - \theta)^2 \int_{\Omega_2} \int_{\Omega_2} {d
          \mu_2(x) \, d \mu_2(y) \over |x - y|} + 2 \theta (1 - \theta)
          \int_{\Omega_2} \int_{\Omega_1}  {d
          \mu_1(x) \, d \mu_2(y) \over |x - y|} \notag \\
          \leq  \theta^2 \I_1(\Omega_1) + (1 - \theta)^2 \I_1(\Omega_2) +
          {2 \theta (1 - \theta) \over \mathrm{dist} \, (\Omega_1,
          \Omega_2)}\,, \notag 
	\end{align}
	which gives \eqref{stuno}.
	
	\smallskip
	
	Let now $\mu$ be the equilibrium measure for
	$\I_1(\Omega_1 \cup \Omega_2)$. Observe that by Lemma \ref{l:optm}
	we have $\mu (\Omega_i) > 0$ for $i = 1,2$, in view of the fact that
	the sets $\Omega_i$ have non-empty interiors. Define the
	probability measures
	$\mu_i:=\mu^{-1}(\Omega_i) \mu|_{\Omega_i}$, and let
	$\bar \theta := \mu(\Omega_1) \in (0,1)$. Then it holds
	\begin{align}
          \I_1(\Omega_1 \cup \Omega_2) = \int_{\Omega_1 \cup \Omega_2}
          \int_{\Omega_1 \cup \Omega_2} {d 
          \mu(x) \, d \mu(y) \over |x - y|} = \mu^2(\Omega_1)
          \int_{\Omega_1} \int_{\Omega_1} {d 
          \mu_1(x) \, d \mu_1(y) \over |x - y|} \\
          + \mu^2(\Omega_2) \int_{\Omega_2} \int_{\Omega_2} {d
          \mu_2(x) \, d \mu_2(y) \over |x - y|} + 2 \mu(\Omega_1)\mu(\Omega_2)
          \int_{\Omega_2} \int_{\Omega_1}  {d
          \mu_1(x) \, d \mu_2(y) \over |x - y|} \notag \\
          > \bar\theta^2 \I_1(\Omega_1) + (1 - \bar\theta)^2
          \I_1(\Omega_2) \,, \notag 
	\end{align}
	which gives \eqref{stdue}.
\end{proof}

The first result in Lemma \ref{l:partit} basically says that the
capacitary energy of a union of two admissible sets may be estimated
from above by the self-energies of the components with an arbitrary
partition of charges, plus an interaction energy in terms of the
(positive) distance between these sets. Similarly, the second result
says that the capacitary energy may be estimated from below by the sum
of the capacitary energies for a suitable distribution of charges.

\subsection{Spherical and elliptical domains}

We continue by summarizing some well-known facts about the Riesz
capacity of balls and elliptical sets.

\begin{lemma}
	\label{l:ell}
	Let
	\begin{align}
	\label{eq:Omell}
	\Omega_e = \left\{ (x, y) \ : \ {x^2 \over a^2} + {y^2 \over b^2} \leq
	1 \right\},
	\end{align}
	be an ellipse with eccentricity $e = \sqrt{1 - (b/a)^2}$, where $a$
	and $b$ are the major and minor semi-axes, respectively. Then
	\begin{align}
	\label{eq:PI1ell}
	\mathcal H^1(\partial \Omega_e) = 4 a \mathrm{E} (e^2) \qquad
	\text{and} \qquad  \I_1(\Omega_e) = a^{-1} \mathrm{K} (e^2), 
	\end{align}
	where $\mathrm{K}(m)$ and $\mathrm{E}(m)$ are the complete elliptic
	integrals of the first and second kind, respectively
	\cite{abramowitz}. In particular, when $a = b$ we have
	$\I_1(\Omega_e) = {\pi \over 2 a}$. The equilibrium measure $\mu$
	for $\I_1(\Omega_e)$ is given by
	\begin{align}
	\label{eq:muell}
	d \mu(x, y) = {1 \over 2 \pi \sqrt{a^2 b^2 - b^2 x^2 - a^2 y^2}}
	\, dx \, dy.
	\end{align}
	In addition, we have
	\begin{align}
	\label{eq:dudko}
	\I_1(\Omega_e) \leq \left( { \pi^5 \over 4 |\Omega_e|
		\mathcal H^1(\partial\Omega_e)}  
	\right)^{1/3}.
	\end{align}
\end{lemma}

\begin{proof}
	Most of the formulas in this lemma are well-known. For the reader's
	convenience, we present some details regarding the Riesz capacitary
	energy of an elliptical set. The value of $\I_1(\Omega_e)$ and
	the expression for the equilibrium measure $\mu$ follow by passing
	to a suitable limit from the corresponding expressions for the
	ellipsoid, which can be found by an explicit solution of the
	Laplace's equation in three dimensions in ellipsoidal coordinates
	\cite{landau8}.
	
	The formula in \eqref{eq:dudko} is motivated by \cite{dudko04},
	where it was shown that
	\begin{align}
	\label{eq:I1OmP}
	\left( { 4 |\Omega_e| \mathcal H^1(\partial \Omega_e)
		\I_1^3(\Omega_e) \over \pi^5} 
	\right)^{1/3} = 1 + O \left( e^8 \right),
	\end{align}
	and the inequality can be seen numerically for any particular value
	of $e \in (0,1)$. To obtain the upper bound on the whole interval is
	not straightforward, however, since many cancellations give rise to
	\eqref{eq:I1OmP}.
	
	The proof of \eqref{eq:dudko} amounts to showing that the function
	\begin{align}
	f(x) := {16 \sqrt{1 - x} \, \mathrm{K}^3(x) \mathrm{E}(x) \over
		\pi^4}, 
	\end{align}
	is bounded above by 1 for all $x \in (0,1)$. For this, we
	differentiate $f$ to obtain, using the identities associated with
	the complete elliptic integrals \cite{abramowitz}:
	\begin{align}
	f'(x) = {8 \mathrm{K}^2(x) (3 \mathrm{E}^2(x) - (2 - x)
		\mathrm{E}(x) \mathrm{K}(x) - (1 - x) \mathrm{K}^2(x)) \over \pi^4 
		x \sqrt{1 - x}}.
	\end{align}
	Therefore, our proof is concluded, once we prove that the function
	\begin{align}
	\label{eq:g}
	g(x) := 3 \mathrm{E}(x) - (2 - x)
	\mathrm{K}(x) - (1 - x) {\mathrm{K}^2(x) \over \mathrm{E}(x)},
	\end{align}
	is negative for all $x \in (0,1)$. This is done by using the
	estimates
	\begin{align}
	\label{eq:KE15}
	\mathrm{K}(x) \geq \mathrm{K}_{15}(x), \qquad \mathrm{E}(x) \leq
	\mathrm{E}_{15}(x), 
	\end{align}
	where $\mathrm{K}_{15}(x)$ and $\mathrm{E}_{15}(x)$ are the
	15th-order Taylor polynomials for $\mathrm{K}(x)$ and
	$\mathrm{E}(x)$, respectively, obtained by truncating the power
	series representations of $\mathrm{K}(x)$ and
	$\mathrm{E}(x)$. Substituting \eqref{eq:KE15} into \eqref{eq:g}, we
	find that $g(x) \leq x^4 P_{27}(x) / Q_{15}(x)$, where $P_{27}(x)$
	and $Q_{15}(x)$ are two explicit polynomials of degrees 27 and 15,
	respectively, with $P_{27}(0) > 0$ and $Q_{15}(0) < 0$.  The proof
	is concluded by verifying that both $P_{27}$ and $Q_{15}$ do not
	change sign in $[0,1]$. All these tedious computations were carried
	out using computer algebra software Mathematica 10.4.
\end{proof}

We now observe that since for a given $\Omega \in \mathcal A_m$ it is
always possible to find $e \in [0, 1)$ such that
$\Omega_e\in \mathcal A_m$ and
$\mathcal H^1(\partial \Omega_e) =\mathcal H^1(\partial \Omega)$, by
comparing the energy of $\Omega$ with that of $\Omega_e$ we obtain the
following result from \eqref{eq:dudko}.

\begin{corollary}
	\label{c:dudko}
	Let $\lambda > 0$, $m > 0$ and let $\Omega \in \mathcal A_m$ be a
	minimizer of $\E^Q_\lambda$ over $\mathcal A_m$. Then
	\begin{align}
	\label{eq:EOmOme}
	\E^Q_\lambda(\Omega) \leq \mathcal H^1(\partial \Omega) +
	\lambda \left( { \pi^5 \over  4 m \mathcal H^1(\partial \Omega)}
	\right)^{1/3}.  
	\end{align}
\end{corollary}

Note that numerically the right-hand side of \eqref{eq:dudko} gives a
remarkably accurate value of $\I_1(\Omega_e)$ for practically all
values of $e$ \cite{dudko04}. For example, when $e = 0.7$, it predicts
the value of $\I_1(\Omega_e)$ with the relative error of
$5 \times 10^{-5}$. Even at $e=0.9999$ the relative error of the
leading order expression is within about 25\%. Let us also point
out that the leading order term in \eqref{eq:dudko} was found to be a
very good approximation for the values of $\I_1$ for sets of varied
simple shapes, not limited to elliptical sets \cite{dudko04}. In fact,
if \eqref{eq:EOmOme} were an equality, one would immediately conclude
from the isoperimetric inequality that balls are the only minimizers
of $\E^Q_\lambda$ over $\mathcal A_m$ for all
$\lambda \leq \lambda_{c2}^Q$, where $\lambda_{c2}^Q$ is given by
\eqref{eq:lamc2}.

\medskip

Recalling Lemma \ref{l:partit}, the computations of Lemma \ref{l:ell}
allow us to make conclusions about instabilities of balls.

\begin{lemma}
  \label{l:lamc1c2}
  Let $m > 0$, let $R = \sqrt{m / \pi}$ and let $\lambda_{c1}^Q$ and
  $\lambda_{c2}^Q$ be as in \eqref{eq:lamc1} and
  \eqref{eq:lamc2}. Then
  \begin{enumerate}[(i)]
  \item If $\lambda > \lambda_{c1}^Q$, a ball with area $m$ is
    unstable with respect to splitting, i.e.,
    $\E^Q_\lambda(\overline B_R) > \E^Q_\lambda(\overline
    B_{R/\sqrt{2}} (0) \cup \overline B_{R/\sqrt{2}} (x_0)$
    for all $x_0 \in \R^2$ with $|x_0| \gg R$.
  \item If $\lambda > \lambda_{c1}^Q$, a ball with area $m$ is
    unstable with respect to elongations, i.e., there exists an
    elliptical set $\Omega_e$ such that
    $\E^Q_\lambda(\overline B_R) > \E^Q_\lambda(\Omega_e)$.
  \end{enumerate}
\end{lemma}

\begin{proof}
  The first assertion follows by a straightforward computation, using
  Lemma \ref{l:ell} and Lemma \ref{l:partit}. For the second, we use
  Corollary \ref{c:dudko}, minimizing the right-hand side of
  \eqref{eq:EOmOme} with respect to
  $\mathcal H^1(\partial \Omega) \geq 2 \pi R$.
\end{proof}

\medskip

We now show that a single ball has lower energy than a disjoint
union of balls with the same total perimeter.

\begin{lemma}\label{lemball}
	Let $\Omega=\bigcup_{i=1}^{N} \overline{B}_{r_i}(x_i)$ be a union of
	a finite number of disjoint {closed} balls, and let
	$r : = \sum_{i=1}^N r_i$.  Then
	\begin{align}
	\label{central}
	\E_\lambda^Q(\overline B_r)\le \E_\lambda^Q(\Omega)\,,
	\end{align}
	with equality if and only if $\Omega=\overline B_r(x)$ for some
	$x\in\R^2$.
\end{lemma}

\begin{proof}
	If $\Omega$ is a ball then there is nothing to prove. 
	
	\smallskip
	Assume that
	$\Omega = \overline B_{r_1}(x_1)\cup \overline B_{r_2}(x_2)$ is the
	union of two disjoint balls.  We show that the ball $\overline B_r$,
	with $r_1+r_2=r$, has always {less energy}.  Indeed, for every
	$\theta \in [0,1]$ it holds
	\begin{align}
	\label{r1r2}
	\frac{1}{r_1+r_2}\le \frac{\theta^2}{r_1}+\frac{(1 - \theta)^2}{r_2},  
	\end{align}
	which is readily verified by noticing that
	$\hat\theta = r_1 / (r_1 + r_2)$ minimizes the right-hand side of
	\eqref{r1r2}, and equality holds for $\theta= \hat\theta$.  Recalling
	that $\I_1(\overline B_{r})=\pi/(2r)$, inequality \eqref{r1r2} can be
	rephrased as
	\begin{equation}\label{eqballs}
	\I_1(\overline B_{r})
	\le \theta^2\I_1(\overline B_{r_1}(x_1))+(1 - \theta)^2\I_1(\overline
	B_{r_2}(x_2))\qquad \forall \theta \in [0,1].
	\end{equation}
	On the other hand, since
	$\mathcal H^1(\partial B_{r})= \mathcal H^1(\partial B_{r_1}(x_1))+
	\mathcal H^1(\partial B_{r_2}(x_2))$,
	from \eqref{eqballs} and \eqref{stdue} it follows that
	\begin{align}\label{sttre}
	\E_\lambda^Q (\overline B_{r}) \leq  \E_{\lambda \bar\theta^2} ^Q 
	(\overline B_{r_1}(x_1)) + \E_{\lambda (1 - \bar\theta)^2} ^Q
	(\overline B_{r_2}(x_2))  <
	\E_{\lambda} ^Q (\overline B_{r_1}(x_1) \cup \overline B_{r_2}(x_2)). 
	\end{align}
	
	\smallskip
	Assume now that $\Omega$ consists of {more than two disjoint}
	balls. In this case we write $\Omega = \Omega' \cup \Omega''$, where
	$\Omega':=\overline B_{r_1}(x_1) \cup \overline B_{r_2}(x_2)$ {and
		$\Omega'':=\Omega\setminus\Omega' \not=\emptyset$. Again, from
		Lemma \ref{l:partit}} we know that there exists
	$\bar \theta\in (0,1)$ and $\delta>0$ such that
	\begin{align}
	\I_1(\Omega)\ge\bar \theta^2\I_1(\Omega')+(1-\bar
	\theta)^2\I_1(\Omega'')+\delta. 
	\end{align}
	Moreover, from the previous discussion it follows that
	$\I_1(\Omega') \ge \I_1(\overline B_{\bar r}(x))$, where
	$\bar r := r_1 + r_2$, for any $x\in\R^2$, so that we get
	\begin{align}%\label{stabat}
	\I_1(\Omega)\ge \bar \theta^2\I_1(\overline B_{\bar r} (x))+(1-\bar
	\theta)^2\I_1(\Omega'')+\delta. 
	\end{align}
	Recalling \eqref{stuno}, we then obtain the estimate
	\begin{align}
          \I_1(\Omega)
          &\ge \bar \theta^2\I_1(\overline B_{\bar r} (x))+(1-\bar 
            \theta)^2\I_1(\Omega'')+\delta \\
          &\ge \I_1(\overline B_{\bar r} (x)\cup \Omega'')+\delta
            -\frac{2\bar\theta(1-\bar\theta)}{{\rm dist}(\overline
            B_{\bar r} (x),\Omega'')}\,. \notag 
	\end{align}
	If we choose $x \in \R^2$ such that
	${\rm dist}(\overline B_{\bar r}(x),\Omega'')>(2\delta)^{-1}$, we then
	get $\I_1(\Omega)>\I_1(\overline B_{\bar r} (x) \cup \Omega'')$.
	Recalling that
	$\mathcal H^1(\partial\Omega)=\mathcal H^1(\partial B_{\bar r}
	(x)\cup \partial\Omega'')$, this implies that
	\begin{align}
	\E_{\lambda} ^Q(\Omega)> \E_{\lambda} ^Q(\overline B_{\bar r} (x)\cup
	\Omega'')\,.   
	\end{align}
	The proof {is then} easily concluded by means of an inductive
	argument.
\end{proof}

Our next lemma is a version of a well-known Brunn-Minkowski inequality
for the Newtonian capacity of convex bodies in $\R^3$
\cite{borell,ColSal,caffarelli}.

\begin{lemma}\label{lemPisa}
	For any pair of convex bodies $\Omega_1,\,\Omega_2$ in $\mathbb R^2$
	and for any $\lambda\in[0,1]$, it holds
	\begin{align}
	C_1(\lambda\Omega_1+(1-\lambda)\Omega_2)\ge\lambda C_1
	(\Omega_1)+(1-\lambda)C_1(\Omega_2),    
	\end{align}
	where $C_1(\Omega):=\I_1^{-1}(\Omega)$ is the 1-Riesz capacity of
	$\Omega$. Moreover the equality holds if and only if $\Omega_1$ is
	homothetic to $\Omega_2$.
\end{lemma}

\begin{proof}
	The inequality has been proved in \cite[Theorem 1.1 (ii)]{NovRuf}
	and can also be seen as the limit case of the corresponding
	inequality for Newtonian capacity of thin cylindrical domains in
	$\R^3$ \cite{borell}.
	
	To get the characterization of the equality case, we introduce
	$u : \R^3 \to \R$, which is the harmonic extension of $v$ in
	\eqref{eq:vriesz} to $\R^3 \backslash (\Omega \times \{0\})$.
	Notice that up to a constant factor the function $u$ is the
	capacitary potential of a convex set
	$\Omega \times \{0\} \subset \R^3$, and the level sets
	$\{x\in \mathbb R^3: u(x)>t\}$ are regular convex bodies in $\R^3$
	for all $t\in (0,1)$ (see \cite[Theorem 1.1 (i)]{NovRuf}). With this
	observation, the final part of the proof of \cite[Theorem 1,
	pp. 476-478]{ColSal} applies without modifications.
\end{proof}

The following is the cornerstone result for the proof of
Theorem \ref{t:A}.  Its proof, in a different setting, was developed
in \cite{bucfralam}.

\begin{lemma}\label{l:balls}
	Let $\Omega\in\mathcal A$ be a convex set. Then
	$\E_\lambda^Q(\Omega)\ge \E_\lambda^Q(\overline B_\Omega)$, where
	$\overline B_\Omega$ is a ball with the same perimeter as $\Omega$,
	with equality if and only if $\Omega$ is a closed ball.
\end{lemma}

\begin{proof}
  By a result of Hadwiger \cite[Theorem 3.3.2]{schneider}, we know
  that there exists a sequence of Minkowski rotation means of $\Omega$
  which converges to a ball in the Hausdorff metric.  Namely, there
  exists a sequence of rotations $\{T_n\}_{n\in\mathbb N}$ such that
  the sets
  \begin{align}
    \Omega_n:=\frac{\sum_{k=1}^n T_k(\Omega)}{n}    
  \end{align}
  converge in the Hausdorff metric to a ball $\overline B_\Omega$. We
  note that by Lemma \ref{lemPisa} we have
  \begin{equation}\label{eq:lem7}
    \I_1(\Omega_n)=\frac{1}{C_1(\Omega_n)}\le\frac{1}{C_1(\Omega)}=\I_1(\Omega).
  \end{equation}
  Suppose that the equality holds in \eqref{eq:lem7} for any
  $n\in\mathbb N$, so that $\Omega_n$ is homothetic to $\Omega$. Since
  $\Omega_n\to \overline B_\Omega$ in Hausdorff metric, we get that
  $\Omega=\overline B_\Omega$.
  % Thus up to passing to a subsequence we can suppose that
  % $\{\I_1(\Omega_n)\}_{n\in\mathbb N}$ is an increasing sequence,
  % and there is at least an $n\in\mathbb N$ such that
  % $\I_1(\Omega_n)>\I_1(\Omega_{n-1})$ unless
  % $\Omega=\overline B_\Omega$.
	
  Since the perimeter is linear with respect to the Minkowski sum of
  convex sets in $\R^2$, we have
  $\mathcal H^1(\partial \overline B_\Omega)=\mathcal
  H^1(\partial\Omega)$.  Gathering the information above, we get
	%  $\I_1$ is sublinear with respect to the Minkowski
	%  summation \cite[Theorem 1.1]{NovRuf}, so that
  \begin{align}
    \mathcal H^1(\partial \overline B_\Omega)+\lambda\I_{1}(\overline
    B_\Omega)
    &
      =\lim_{n\to\infty} \left( \mathcal H^1(\partial
      \Omega_n)+\lambda \I_{1}(\Omega_n) \right) \le \mathcal H^1(\partial 
      \Omega)+\lambda \I_1(\Omega),
  \end{align}
  with equality if and only if $\Omega=\overline B_\Omega$. This
  concludes the proof.
\end{proof}

\subsection{Indecomposable
  components} \label{indecomposablecomponents}

Given a set $\Omega \in \mathcal A$, let $\overline \Omega^M$ be its
measure theoretic closure, namely:
\begin{align}
\label{eq:OmMbar}
\overline \Omega^M := \left\{ x \in \R^2 \ : \ \liminf_{r \to 0}
{|\Omega \cap B_r(x)| \over \pi r^2} > 0 \right\}.
\end{align}
Since
$P(\overline{\Omega}^M)= P(\Omega) = \mathcal H^1(\partial^M \Omega)
\leq \mathcal H^1(\partial \Omega) < +\infty$,
the set $\overline \Omega^M$ is a set of finite perimeter \cite{AFP}.
Therefore, following \cite{ACMM}, for a set $\Omega \in \mathcal A$ we
can write
$\overline \Omega^M = \left( \bigcup_{{i}}\Omega_i\right) \cup
\Sigma$,
with $\mathcal H^1(\Sigma) = 0$, where the sets
$\Omega_i\in \mathcal A$ are the so-called {\it indecomposable
	components} of $\overline\Omega^M$ (finitely many or a countable
family). In particular, the sets $\Omega_i$ admit unique
representatives that are connected and satisfy the following
properties:
\begin{itemize}
	\item $\mathcal H^1(\Omega_i\cap \Omega_j)=0$ for $i\ne j$, 
	\item $|\overline\Omega^M|=\sum_{{i}} |\Omega_i|$,
	\item
	$P(\Omega)=P(\overline\Omega^M)
	=\sum_{{i}}P(\Omega_i)$,
	\item $\Omega_i=\overline{\text{int}(\Omega_i)}$.
\end{itemize} 
Moreover, each set $\Omega_i$ is indecomposable in the sense that it
cannot be further decomposed as above. We refer to these
representatives of $\Omega_i$ as the {\em connected components} of
$\Omega \in \mathcal A$. Note that such a definition is necessary,
since we are dealing with equivalence classes of subsets of $\R^2$ for
which the usual notion of connectedness is not well defined (consider,
for instance, the union of two touching closed balls, which is
equivalent to the same set minus the contact point). We point out,
however, that this notion coincides with the standard notion of
connected components in the following sense: if $\Omega$ has a regular
boundary, for instance, Lipschitz continuous, then the components
$\Omega_i$ are (the closure of) the usual connected components of the
interior of $\Omega$.

Observe that
$\mathcal H^1(\partial \, co(\Omega_i))\le \mathcal H^1(\partial
\Omega_i)$
for every $i$, where $co(\Omega_i)$ denotes the convex envelope of
$\Omega_i$. This follows from the fact that the outer boundary of a
connected component can be parametrized by a Jordan curve of finite
length (see \cite[Section 8]{ACMM}). In addition, since
$\partial \Omega$ is negligible with respect to the equilibrium
measure for $\I_1(\Omega)$ by Lemma \ref{l:optm}, we have
$\I_1(\Omega) = \I_1(\overline \Omega^M)$.  Lastly, we have
$|\Omega| = |\overline \Omega^M|$ and
$\mathrm{int}(\Omega) = \mathrm{int} (\overline \Omega^M)$.

We conclude this section by showing that any admissible set may be
replaced by an admissible set with lower energy that consists of only
a finite number of connected components. In particular, this result
shows that minimizers of $\E^Q_\lambda$ over $\mathcal A$ are
connected whenever they exist. We note that replacing an admissible
set by a set with finitely many connected components as in the lemma
below imparts a certain degree of regularity. Namely, our procedure
removes the parts of the topological boundary of an admissible set
that have postive $\mathcal H^1$ measure, but do not belong to the
measure theoretic closure (e.g., a line segment attached to a ball).
\begin{lemma}\label{l:components}
  Let $\Omega\in\mathcal A$. Then there exists a set
  $\Omega'\in\mathcal A$ such that $\Omega'$ has a finite number of
  {connected} components, and
  \begin{equation}
    \label{energy}
    \E_\lambda^Q(\Omega')\le\E_\lambda^Q(\Omega).
  \end{equation} 
  Moreover, the inequality in \eqref{energy} is strict whenever
  $\Omega$ has more than one connected component.
\end{lemma}

\begin{proof}
  Without loss of generality, we may assume that $\overline \Omega^M$
  has at least two indecomposable components.  Let
  $\overline{\Omega}^M=(\bigcup_{i\in I}\Omega_i)\cup\Sigma$ be the
  decomposition of $\overline\Omega^M$ into indecomposable components.
  Let $0<\eps<{1/2}$ to be fixed later, and let $N \geq 2$, depending
  on $\eps$, be such that
  $1-\delta:=\sum_{i=1}^{N}\mu(\Omega_i)\ge1-\eps$, where $\mu$ is the
  {equilibrium} measure of $\Omega$.  If $I$ is finite, we let $N$ be
  the number of indecomposable components of $\Omega$.
	
  For $i\in I$, we let $\xi_i$ be the barycenter of $\Omega_i$ and set
  \begin{align}
    \Omega^N(t):=\bigcup_{i=1}^N \left(\Omega_i-\xi_i+{i t e_1}
    \right), 
  \end{align}
  where $e_1$ denotes the unit vector in the first coordinate
  direction.  Let also $f_i^{t}({x}):={x} -\xi_i+{i t e_1}$ and let
  $\mu_i^t:=(f_i^{t})_\sharp \mu|_{\Omega_i}$ be the pushforward of
  $\mu_i:=\mu|_{\Omega_i}$ through the map $f_i^{t}$.  By a
  straightforward computation, we get
  \begin{align}
    \I_1(\Omega)-\I_1(\Omega^N(t))
    &\ge\left(1-\frac{1}{(1-\eps)^2}
      \right)\sum_{i=1}^N\iint \frac{d\mu_i(x)\,d\mu_i(y)}{|x-y|}
    \\
    &+{2 \sum_{i=1}^{N-1} \sum_{j=i+1}^N}
      \iint\frac{d\mu_i(x)\,d\mu_j(y)}{|x-y|} \notag \\
    &-\frac{{2}}{(1-\eps)^2}{\sum_{i=1}^{N-1}
      \sum_{j=i+1}^N}
      \iint\frac{d\mu_i^t(x)\,d\mu_j^t(y)}{|x-y|}\,. \notag 
  \end{align}
  Since $\Omega\in\mathcal A$, there exists $R>0$ such that
  $\diam( \Omega)<R$. Moreover if $t>2R$ we have that whenever
  $i\ne j$, for any
  $(x,y)\in\left(-\xi_i+\Omega_i+{i t e_1} \right)\times\left(-\xi_j
    +\Omega_j+{j t e_1}\right)$ it holds $|x-y|\ge t-2R$.  Let also
  \begin{align}
    M(\Omega):={2 \sum_{i=1}^{N-1} \sum_{j=i+1}^N} \iint
    d\mu_i(x)\,d\mu_j(y)>0.
  \end{align}
  The last inequality follows from Lemma \ref{l:optm}.  With these
  choices, we get
  \begin{align}
    \I_1(\Omega)-\I_1(\Omega^N(t))
    &\ge -{8} \eps
      \sum_{i=1}^N\iint
      \frac{d\mu_i(x)\,d\mu_i(y)}{|x-y|}+\left(\frac{1}{R}
      -{\frac{4}{t-2R}}\right)M(\Omega) 
      \notag 
    \\
    &\ge -{8} \eps\I_1(\Omega) +
      \left(\frac{1}{R}-{\frac{4}{t-2R}}\right)M(\Omega)
      \,\ge\, \frac{M(\Omega)}{{4}R}\,,
  \end{align}
  where the last inequality follows if we choose
  % If we choose $\eps<\frac15$, so that $1-(1-\eps)^{-2}<\eps$, we
  % conclude since for
  $t \geq 10 R$ and $\eps$ small enough depending on $\Omega$.
  % the previous quantity becomes positive.
  Since
  $\mathcal H^1(\partial\Omega)\ge \mathcal H^1(\partial\Omega^N(t))$,
  we get that $\Omega':=\Omega^N(t)$ is the desired set.
\end{proof}

\section{Proofs of the main results}\label{sec:proofs}

We turn now to the proofs of Theorems \ref{t:A}$-$\ref{t:EUAm}.

\begin{proof}[Proof of Theorem \ref{t:A}]
  Let $\Omega\in\mathcal A$, hence $\E_\lambda^Q(\Omega)<+\infty$.  We
  first show that there exists a closed ball $\overline B$ such that
  \begin{equation}\label{eqqom}
    \E_\lambda^Q(\Omega)\ge \E_\lambda^Q(\overline B),
  \end{equation}
  and equality holds if and only if $\Omega$ itself is a ball.
	
  Thanks to Lemma \ref{l:components} we can suppose that the number of
  the {connected} components $\Omega_i$ {of $\Omega$} is finite. If
  $\Omega$ has only one connected component, then we can convexify it
  and use Lemma \ref{l:balls}. So we suppose that $\Omega$ has at
  least two connected components.
	
  Let $co(\Omega_{i})$ be the convex envelope of $\Omega_{i}$, and
  denote by $\mu$ the equilibrium measure of $\Omega$ and $\mu^i$ its
  restriction to $\Omega_i$. Note that $\mu^i(\Omega_i) > 0$ for all
  $i$ by Lemma \ref{l:optm}. Then by Lemma \ref{l:balls} we get
  \begin{equation}\label{2-0}
    \begin{aligned} 
      \E_\lambda^Q(\Omega)&=\sum_i \mathcal H^1(\partial
      \Omega_{i})+\lambda\sum_i \iint
      \frac{d\mu^i{(x)}\,d\mu^i{(y)}}{|x-y|}+\lambda\sum_{i\ne
        j}\iint \frac{d\mu^i{(x)}\,d\mu^j{(y)}}{|x-y|}\\
      &\ge \sum_i\mathcal H^1(\partial \,
      co(\Omega_{i}))+\lambda\sum_i
      {(\mu^i(\Omega_{i}))^2}\I_1(\Omega_{i})+\frac{\lambda
        M(\Omega)}{{\rm diam}(\Omega)}\\
      &\ge \sum_i\mathcal H^1(\partial \,
      co(\Omega_{i}))+\lambda\sum_i
      {(\mu^i(\Omega_{i}))^2}\I_1(co(\Omega_{i}))+\frac{\lambda
        M(\Omega)}{{\rm diam}(\Omega)}\\
      &\ge \sum_i\mathcal H^1(\partial B^i))+\lambda\sum_i
      (\mu^i(\Omega_{i}))^2\I_1(\overline{B^i})+\frac{\lambda
        M(\Omega)}{{\rm diam}(\Omega)}.
    \end{aligned}
  \end{equation}
  where $M(\Omega) := \sum_{i\ne j}\mu(\Omega_i)\mu(\Omega_j)>0$ and
  $B^i$ are balls with mutual distance greater than a number $n$ to be
  fixed later, and such that
  $\mathcal H^1(\partial B^i)=\mathcal H^1(\partial \Omega_{i})$.  In
  the first inequality we used the fact that
  $(\mu^i(\Omega_i))^{-1} \mu^i$ is admissible in the definition of
  $\I_1(\Omega_i)$, while for the second one we exploited the
  inequality $\I_1(\Omega_i)\ge\I_1(co(\Omega_i))$.  Let now
  $\widetilde{\mu}^i$ be the optimal measure for $B^i$. Then the
  measure $\sum_i \mu(\Omega_i)\widetilde{\mu}^i$ is a probability
  measure on $\cup_i B^i$. Thus
  \begin{equation}\label{1-0}
    \begin{aligned}
      \I_1(\cup_i B^i)&\le\sum_i\mu^2 (\Omega_i)\iint
      \frac{d\widetilde{\mu}^i(x)\,d\widetilde{\mu}^i(y)}{|x-y|}+\sum_{i\ne
        j}\mu(\Omega_i)\mu(\Omega_j)\iint
      \frac{d\widetilde{\mu}^i(x)\,d\widetilde{\mu}^j(y)}{|x-y|}\\
      &\le\sum_i \mu^2 (\Omega_i) \, \I_1(B^i)+\frac{\sum_{i\ne
          j}\mu(\Omega_i)\mu(\Omega_j)}{n}.
    \end{aligned}
  \end{equation}
  By \eqref{2-0} and \eqref{1-0} we get
  \begin{align}
    \E_\lambda^Q(\Omega)\ge \E_\lambda^Q(\cup_i B^i)+\frac{\lambda
    M(\Omega)}{{\rm
    diam(\Omega)}}-\frac{\lambda
    M(\Omega)}{n}\ge\E_\lambda^Q(\cup_i B^i),   
  \end{align}
  where the last inequality holds if we choose $n$ big enough,
  depending on $\Omega$.
	
  To conclude, we only have to show that
  \begin{equation}\label{ballminimal}
    \E_\lambda^Q(\cup_i \overline{B^i})\ge \E_\lambda^Q(\overline B)
  \end{equation}
  for a ball $B$ such that
  $\mathcal H^1(\partial B) = \sum_i \mathcal H^1(\partial B^i)$,
  where equality holds only if the union of balls in the left-hand
  side of \eqref{ballminimal} is indeed a unique ball. But this is
  exactly the statement of Lemma \ref{lemball}.  Once we know that the
  minimizer of $\E_\lambda^Q$ in $\mathcal A$ is a ball, we conclude
  by optimizing on the radius of all possible balls.
\end{proof}

\begin{proof}[Proof of Corollary \ref{t:B}]
  The proof is contained in that of Theorem \ref{t:A}, also noting
  that at every step of the proof the perimeters of the considered
  sets cannot increase. Therefore, if the perimeter of the final ball
  $B$ is strictly less than $\mathcal H^1(\partial \Omega)$, we can
  further decrease $\I_1$ by dilating that ball to ensure that its
  perimeter equals $\mathcal H^1(\partial \Omega)$.
\end{proof}

\begin{proof}[Proof of Theorem \ref{t:Am}]
  We first prove $(i)$.  Notice that if $\lambda=\lambda_0^Q$, by
  Theorem \ref{t:A} we know that the ball $\overline B_R$ is a global
  minimizer of $\mathcal E_{\lambda}^Q$ in $\mathcal A$.  In
  particular
  $\E^Q_{\lambda_0^Q} (\Omega)\ge \E^Q_{\lambda_0^Q}(\overline B_R)$,
  with equality if and only if $\Omega$ is a closed ball.
	
  Now, since $\mathcal \I_1$ is maximized by balls under volume
  constraint, see \cite[p. 14]{polya1} and \cite[Sec. VII.7.3,
  p. 157]{polya2}, and since $\lambda_0^Q - \lambda>0$, we get
  \begin{align}
    \E^Q_\lambda (\Omega) = \E^Q_{\lambda_0^Q} (\Omega) - (\lambda_0^Q -
    \lambda) \I_1(\Omega) \geq \E^Q_{\lambda_0^Q}(\overline B_R) -
    (\lambda_0^Q - \lambda) \I_1(\Omega) \\
    \geq \E^Q_{\lambda_0^Q}(\overline B_R) -
    (\lambda_0^Q - 
    \lambda) \I_1(\overline B_R) = \E^Q_\lambda(\overline B_R), \notag
  \end{align}
  with equality if and only if $\Omega$ is a ball; thus
  $\overline B_R$ is a minimizer. Moreover all other minimizers are
  translates of $\overline B_R$. Indeed every minimizer has $B_R$ as
  its interior (up to a translation), and the boundary of the
  minimizer differs from $\partial B_R$ by a set of zero
  $\mathcal H^1$ measure.
	
  \medskip
	
  We now prove $(ii)$.  We have to show that there is no minimizer of
  $\E^Q_\lambda$ over $\mathcal A_m$ for $\lambda > \lambda_0^Q$.  We
  argue by contradiction and assume that there is a minimizer
  $\Omega \in \mathcal A_m$ of $\E^Q_\lambda$. In view of the
  statement of Theorem \ref{t:A}, we have
  $\E^Q_\lambda(\Omega) > 2 \pi \sqrt{\lambda}$. We now consider a
  competitor set
  $\widetilde \Omega = \overline B_{\widetilde R}(0)
  \cup\left(\bigcup_{i=1}^N \overline B_{r/N} (x_i)\right)$,
  where $N \in \mathbb N$, $r > 0$ is to be determined,
  ${\widetilde R} = \sqrt{ {m \over \pi} - {r^2 \over N}}$, and
  $x_i \in \R^2$ with $|x_i - x_j| \gg {\widetilde R}$ for any
  $i \not= j$. As in Lemma \ref{l:partit}, the energy of
  $\widetilde \Omega$ may be estimated from above by partitioning the
  total charge between $\overline B_{\widetilde R}(0)$ and
  $\overline B_{r/N} (x_i)$.  Suppose that the {fraction of the total}
  charge on each {ball} $\overline B_{r/N} (x_i)$ is equal to
  {$\theta / N$ for some} $\theta \in (0,1)$. We chose
  $r = {\frac12 \theta \sqrt{\lambda}}$ and $\theta$ in such a way
  that the charge {fraction $1 - \theta$} on
  $\overline B_{\widetilde R} (0)$ is optimal for a ball of radius
  ${\widetilde R}$ according to Theorem \ref{t:A}: Then we get
  \begin{align}
    \lambda ( 1 - \theta)^2 = {4 \widetilde R^2 = } \lambda_0^Q -
    {\theta^2 \lambda \over N},  
  \end{align}
  which for $\lambda > \lambda_0^Q$ has a unique solution for $\theta$
  whenever $N$ is sufficiently large. On the other hand, by Lemma
  \ref{l:partit} with our choices of the {charge partitioning} we
  have, for any $\eps > 0$ and $x_i$ sufficiently far apart:
  \begin{align}
    \E^Q_\lambda(\widetilde \Omega) 
    & \leq \E^Q_{\lambda(1 - \theta)^2} (\overline
      B_{R_0}(0)) + \sum_{i=1}^N \E^Q_{\lambda \theta^2 / N^2}
      (\overline B_{r/N} (x_i)) + \eps \\
    & = 2 \pi (1 - \theta) \sqrt{\lambda} + \sum_{i=1}^N {2 \pi \theta
      \sqrt{\lambda} \over N} + \eps =  2 \pi 
      \sqrt{\lambda} + \eps. \notag 
  \end{align}
  In view of arbitrariness of $\eps$, this contradicts the minimality
  of $\E^Q_\lambda(\Omega) > 2 \pi \sqrt{\lambda}$.
\end{proof}

\begin{proof}[Proof of Theorem \ref{t:scale}]
  Just use as a competitor the set $\widetilde{\Omega}$ constructed in
  the proof of Theorem \ref{t:Am}, point $(ii)$, to show that for any
  $\varepsilon>0$ there exists a set whose energy is lower than
  $2\pi\sqrt{\lambda}+\varepsilon$. By Theorem \ref{t:A} its energy is
  greater than $2\pi\sqrt{\lambda}$, thus we conclude.
\end{proof}

\begin{proof}[Proof of Theorem \ref{t:EUA}]
  By Corollary \ref{t:B} we know that for all $\Omega\in\mathcal A$ we
  have $\E^U_\lambda(\Omega)\ge \E^U_\lambda(\overline B_\Omega)$,
  where $\overline B_\Omega$ is a ball with the same perimeter as
  $\Omega$, with equality if and only if $\Omega$ is a ball.  As the
  energy of a ball of radius $R$ is given in \eqref{eq:EUBR}, this
  immediately implies that if $\lambda=\lambda_0^U$, then the energy
  of any ball is zero. In particular, statement $(ii)$ follows.
	
  If $\lambda>\lambda_0^U$, by \eqref{eq:EUBR} the energy of a ball of
  radius $R$ diverges to $-\infty$ as $R\to+\infty$, proving statement
  $(iii)$.  To prove $(i)$, let $\lambda<\lambda_0^U$ and recall that
  $\I_1^{-1}$ is minimized by balls under volume constraint. Then we
  have that, for any $\Omega\in\mathcal A$, it holds
  \begin{align}
    \E_\lambda^U(\Omega)
    &= \E_{\lambda_0^U}^U(\Omega)+(\lambda_0^U-\lambda)
      \I_1^{-1}(\Omega) 
      \ge \E_{\lambda_0^U}^U(\overline
      B_R)+(\lambda_0^U-\lambda) \I_1^{-1}(\overline
      B_R) 
    \\ \label{turchia-croazia}
    &=\E_\lambda^U(\overline B_R)=
      (\lambda_0^U-\lambda) \frac{2R}{\pi}>0, \notag 
  \end{align}
  where $\overline B_R$ is a ball with the same volume as $\Omega$.
  Therefore, the infimum of $\E_\lambda^U$ is zero, and it is not
  attained.
\end{proof}

\begin{proof}[Proof of Theorem \ref{t:EUAm}]
  By \eqref{turchia-croazia} and Theorem \ref{t:EUA} $(ii)$ we
  immediately obtain $(i)$.  Let us prove $(ii)$. For $n\in\mathbb N$
  and $d>0$ {sufficiently large}, we let $r_n:=\sqrt{m/(\pi n)}$ and
  \begin{align}
    \Omega_{n,d}:=\bigcup_{i=1}^n B_{r_n}\left(x_i\right),    
  \end{align}
  where the centers $x_i\in \R^2$ are chosen in such a way that
  $|x_i-x_j|\ge d$ for $i\ne j$.  Thus $|\Omega_{n,d}|=m$, and {by
    Lemma \ref{l:partit}} its asymptotic energy as $d\to \infty$ {can
    be estimated as}
  \begin{align}
    {\limsup}_{d\to \infty}\E_\lambda^U(\Omega_{{n,d}}) {\leq}
    n\,\E_\lambda^U(\overline B_{r_n})
    =2\sqrt{\pi m}\left(1-\frac{\lambda}{\pi^2}
    \right)\sqrt{n}.      
  \end{align}
  Since $\lambda>\pi^2$, the latter diverges to $-\infty$ as
  $n\to \infty$.
\end{proof}

%%      ---------------------------------------------------------------------
%%      ------------------------- APPENDIX (OPTIONAL) -----------------------
%%      ---------------------------------------------------------------------
        
%%      If you have one appendix, uncomment the line \appendix and add
%%      a \section{ *** APPENDIX TITLE ***}. If you have more than
%%      one, uncomment the line \appendices and add a \section{ ***
%%      APPENDIX TITLE ***} command for each appendix title.

%\appendix
%\appendices
%\section{}

%%      Type body of appendix/-ices here.

%%      ---------------------------------------------------------------------
%%      ---------------------------ACKNOWLEDGMENTS (OPTIONAL) ---------------
%%      ---------------------------------------------------------------------

%% ***** UNCOMMENT THE FOLLOWING LINE TO ADD ACKNOWLEDGMENTS.

 \ack 

%%      Type acknowledgments here.
 The authors wish to thank M. Goldman, J. Lamboley and V. Moroz for
 helpful discussions on the subject. The work of CBM was supported, in
 part, by NSF via Grants No. DMS-1313687 and DMS-1614948. MN and BR
 were partially supported by the University of Pisa via grant
 PRA-2015-0017 and by GNAMPA of INdAM.
%%      ---------------------------------------------------------------------
%%      --------------------------- BIBLIOGRAPHY ----------------------------
%%      ---------------------------------------------------------------------

\frenchspacing
\bibliographystyle{cpam}
\bibliography{conductdrop2_cpam}

\end{document}